\documentclass[12pt]{article}

\usepackage{color}
\usepackage{times}
\usepackage[pdftex]{hyperref}
\usepackage{float}

\usepackage{enumitem}
\usepackage[utf8]{inputenc}
\usepackage[english]{babel}
\usepackage{graphicx}
\usepackage{caption}
\usepackage{amsthm,amssymb}
\usepackage{amsmath}
\usepackage{thmtools}
\usepackage{mathrsfs}
\usepackage{bbm}
\usepackage{pdfpages}
\usepackage{enumitem}
\usepackage{natbib}
\usepackage{authblk}

\newcommand{\R}{\mathbb{R}}

\newcommand{\Z}{\mathbb{Z}} %ring of integers

\theoremstyle{plain}
\newtheorem{theorem}{Theorem}
\newtheorem{lemma}{Lemma}

\theoremstyle{definition}
\newtheorem{definition}{Definition}

\begin{document}

\begin{title}
	{\Large\bf An upper bound on the per-tile entropy of ribbon tilings}
\end{title}
\renewcommand\Affilfont{\itshape\small}
\author[1]{Simon Blackburn}
\author[2]{Yinsong Chen}
\author[3]{Vladislav Kargin}
\affil[1]{Department of Mathematics, Royal Holloway University of London \protect\\ Egham, Surrey TW20 0EX, United Kingdom.
\protect\\ s.blackburn@rhul.ac.uk}
\affil[2]{Department of Mathematics and Statistics, Binghamton University \protect\\ Binghamton, New York, U.S.A.
\protect\\ ychen276@Binghamton.edu}
\affil[3]{Department of Mathematics and Statistics, Binghamton University \protect\\ Binghamton, New York, U.S.A.
\protect\\ vkargin@Binghamton.edu}
\def\RunningHead{}

\maketitle

\begin{abstract}
This paper considers $n$-ribbon tilings of general regions and their per-tile entropy (the binary logarithm of the number of tilings divided by the number of tiles). We show that the per-tile entropy is bounded above by $\log_2 n$. This bound improves the best previously known bounds of $n-1$ for general regions, and the asymptotic upper bound of $\log_2 (en)$ for growing rectangles, due to Chen and Kargin.
\end{abstract}

\section{Introduction}
A \emph{region} $R \in \R^2$ is a union of finite number of unit squares $[x, x+1] \times [y, y+1]$, with $(x, y) \in \Z^2$. Two unit squares are \emph{adjacent} if they share the same edge. We consider tilings of regions $R$ by ribbon tiles:

\begin{definition} 
	A \emph{ribbon tile} of length $n$, or an \emph{$n$-ribbon}, is a connected sequence of $n$ unit squares in $\R^2$, each of which (except the first one) comes directly above or to the right of its predecessor.
\end{definition}

For $(x,y)\in\Z^2$, we say a square $[x, x+1] \times [y, y+1]$ has \emph{level} $x+y$.
An $n$-ribbon can also be defined as a connected set of squares containing exactly one square of each of $n$ (consecutive) levels. See Figure \ref{FigTilingOrder4} for an illustration of the case $n = 4$. We call the first square of an $n$-ribbon (the square with the smallest level) the \emph{root square}, and the last square of an $n$-ribbon (the square with the largest level) the \emph{end square}.
         
\begin{figure}[H] 
	\centering
	\includegraphics[scale=0.35]{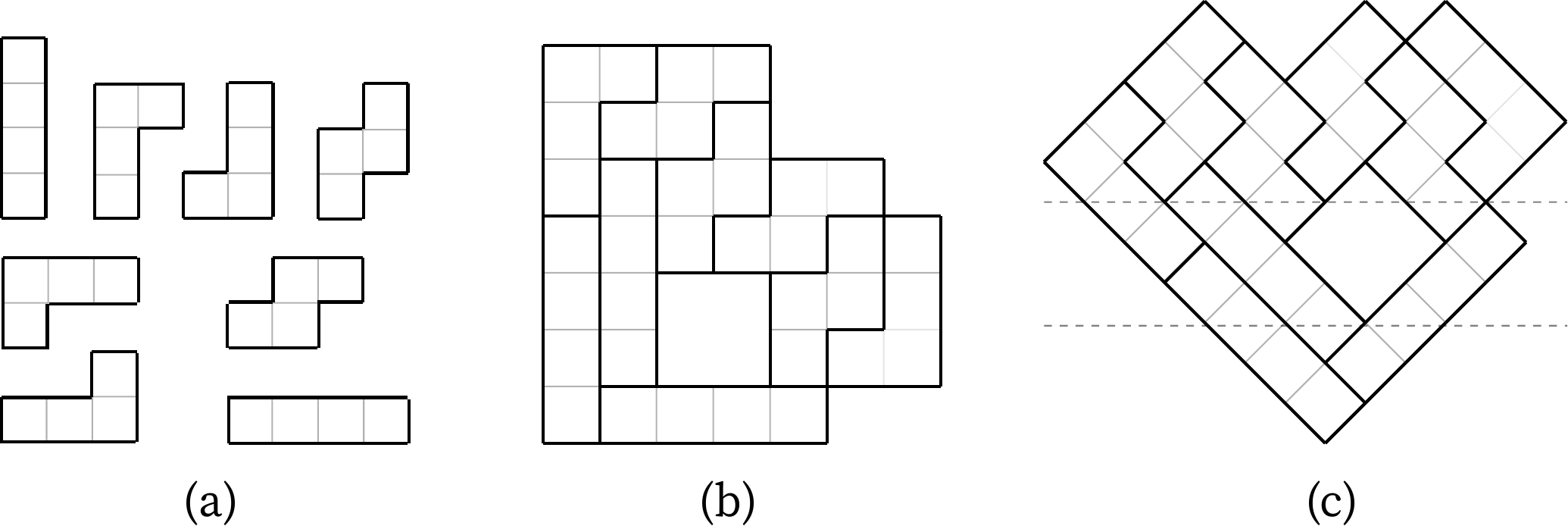} 
	\caption{$n=4$. Picture (a) shows eight types of ribbon tiles. Picture (b) shows a tiling of a region with a hole. Picture (c) is obtained by rotating picture (b) forty-five degrees counter-clockwise. If the lowest square has level $0$, the horizontal lines pass through squares of levels $2$ and $5$ respectively.}  
	\label{FigTilingOrder4}
\end{figure}

Dominoes are a particular case of ribbon tiles with $n = 2$, and have been extensively studied. Ribbon tilings of Young tableaux, known as rim hook tableaux, have also had attention as part of the representation theory of the symmetric group. (See \citep{stanley2002rank},  \citep{borodin1999longest}, \citep{fomin1997rim}, \citep{james1987representation}(Chapter 1) and \citep{stanton1985schensted}.) The $n$-ribbon tilings with $n > 2$ for more general regions were first studied in \cite{pak2000ribbon}.  

Typical questions one might ask about tilings are: Does a tiling of a region $R$ exist? If so, can these tilings be enumerated?

The existence question for $n$-ribbon tilings of regions that are simply connected was studied in \cite{sheffield2002ribbon}, who provided a remarkable algorithm, linear in the area of a region, that checks whether the region has an $n$-ribbon tiling. The existence question for general (not necessarily simply connected) regions is still open, but might be hard: \cite{akagi2020hard} showed that for general regions, the existence of tilings by $180$-trominoes is an NP-complete decision problem.

In this paper, we focus on the enumeration question. For domino tilings, this question has been widely studied. The papers \cite{kasteleyn1961statistics} and \cite{temperley1961dimer} provide a formula for the number of tilings of rectangular regions using a method based on calculation of Pfaffians. For rectangular regions of fixed height, \cite{klarner1980domino} used another method (a difference equation method) for enumeration, and \cite{stanley1985dimer} studied the properties of the generating function. For regions that are Aztec diamonds, the enumeration problem was solved by \cite{elkies1992alternating}.

In contrast, much less is known about the enumeration of $n$-ribbon tilings when $n > 2$. For the rest of this introduction, we express any enumeration results in terms of the following quantity.

\begin{definition} 
	The \emph{per-tile entropy} of the $n$-ribbon tilings of a region $R$ is the binary logarithm of the number of $n$-ribbon tilings divided by the number of ribbons in each tiling, that is,
$$
	\mathrm{Ent}_n(R) = \frac{\log_2(\mathcal{T}_n(R))}{A(R)/n}
$$
where $\mathcal{T}_n (R)$ is the number of $n$-ribbon tilings of the region $R$, and $A(R)$ is the area of $R$.
\end{definition}
So the per-tile entropy expresses the average number of possibilities for the position and type of a tile in an $n$-ribbon tiling of $R$.

For $n$-ribbon tilings, it is not difficult to calculate that the number of tilings of an $n\times n$ square region is $n!$. (See \cite{kargin2023enumeration} Lemma 1.) If $n \to \infty$, then the per-tile entropy is $\log_2(n!)/n \sim \log_2 n - \log_2 e$. In \cite{alexandersson2018enumeration}, an exact formula for the number of $n$-ribbon tilings of an $n \times 2 n$ rectangle was proved. The formula implies that for $n \to \infty$, the per-tile entropy is asymptotically equal to  $\log_2 n - \log_2 e + 1 - \frac{1}{2} \log_2 C$ where $C \approx 2.4969$ is a constant associated to Bessel functions (See formula (0.9) in \cite{kaufmann1996higher}). Observe that changing the $n \times n$ square to the $n \times 2 n$  rectangle leads to a significant increase (namely $1 - \frac{1}{2} \log_2 C \approx 0.3399$) in the asymptotic per-tile entropy.

In the above examples, $n$ grows with the size of the region. Upper and lower bounds for $\mathrm{Ent}_n(R)$ are known for various classes of regions when $n$ is fixed. For example, \cite{chen2023number} study tilings of \emph{strips}, which are rectangles of fixed height equal to $n$. It is shown that the per-tile entropy for strips is bounded above by $\log_2 n$. In \cite{kargin2023enumeration}, exact enumerations of $n$-ribbon tilings are provided for two classes of regions, \emph{generalized Aztec diamonds} and \emph{stairs}. The per-tile entropy of generalized Aztec diamonds equals $1/2$. When $n$ is odd, the per-tile entropy of stairs converges  to $\log_2(n + 1) - 1$ as the size of the region grows. By considering a generalisation of stairs, \cite{bevan2024} shows that (for any positive integer $n$) there are regions whose per-tile entropy converges to $\log\lceil n/2\rceil$.  \cite{kargin2023enumeration} also considered the case of rectangular regions where the ribbon length is fixed and both the height and width of the rectangle go to infinity). In this situation, it was shown that the per-tile entropy $\mathrm{Ent}_n(\mathrm{rectanges})$ converges to a limit $\mu^{(n)}$ that satisfies the inequality
$$
\log_2 n  - \log_2 e 
+ n^{-1}\Big(
\frac{1}{2} \log_2 n + \log_2 \sqrt{2\pi}
\Big) \leq \mu^{(n)}\leq \log_2 n + \log_2 e.
$$

The main result of this paper is as follows. 

\begin{theorem}	\label{ThmUpBound}
For every finite region $R$ and every $n \geq 2$, the per-tile entropy of $R$ satisfies the inequality
	\begin{equation}	\label{equ_bound}
		\mathrm{Ent}_n(R) \leq \log_2 n.
	\end{equation}
\end{theorem}

This bound significantly improves the previously best known upper bound for $\mathrm{Ent}_n(R)$ for general regions of $n-1$, established by~\cite{kargin2023enumeration}. Moreover, the examples above (in particular the enumeration results for stairs) show that the upper bound of Theorem~\ref{ThmUpBound} is close to being tight. 

The key observation used to prove the new bound is Lemma \ref{LemRootSquareDetermination}, which shows that a tiling of an arbitrary finite region is determined uniquely by the positions of the root squares of the ribbon. It follows that the number of $n$-ribbon tilings is equal to the number of valid choices of the root squares. The inequality (\ref{equ_bound}) in the main result is proved by an analysis of the constraints on possible choices of the root squares in a tiling. 

We end this introduction by stating an open problem. Let $n$ be a fixed integer with $n\geq 2$. For each integer $A=an$ that is divisible by $n$, choose a region $R_{a}$ of area $A$ with the largest number $t_a$ of $n$-ribbon tilings. Define $s_a=\log_2 t_a$, so $s_a/a$ is the largest $n$-ribbon entropy of a region of area $an$. For positive integers $a$ and $b$, we see that $s_a+s_b\leq s_{a+b}$ (because the region that is the disjoint union of $R_a$ and $R_b$ has at least $t_at_b$ ribbon tilings) and so the sequence $(s_a)$ is superadditive. Theorem~\ref{ThmUpBound} shows that $s_a/a$ is bounded above by $\log_2n$. Hence Fekete's lemma for superadditive sequences implies that $\lim_{a\rightarrow\infty} s_a/a=\lim_{a\rightarrow\infty} \mathrm{Ent}_n(R_a)$ exists (and is equal to $\sup\{s_a/a:a\in\mathbb{N}\}$). So we may define $\tau_n=\lim_{a\rightarrow\infty} s_a/a$. We ask: what is the value of $\tau_n$? From~\cite{bevan2024} and Theorem~\ref{ThmUpBound} we can see that
\[
\log_2 n-1\leq \tau_n\leq \log_2 n,
\]
but we conjecture that neither bound is tight. As a possibly more accessible problem, we ask: does $\tau_n-\log_2n$ converge to a constant as $n\rightarrow\infty$? 

In the remainder of the paper we provide some preliminaries in Section~\ref{SecPreliminaries} and prove our results in Section~\ref{SecProof}.

\section{Preliminaries}
\label{SecPreliminaries}
We say that a tile has \emph{level $l$} if its root square has level $l$. We call a square a \emph{boundary square} if it is not contained in $R$ but adjacent to at least one square in $R$. Let $\mathbf{S}^{(l)}$ and $\mathbf{B}^{(l)}$ be the set of squares and boundary squares of a region $R$ at level $l$, respectively. 

The paper \cite{sheffield2002ribbon} introduced a `\emph{left-of}' relation for both tiles and squares (including boundary squares), denoted by $\prec$. Let $S_{x,y}$ be a square (or a boundary square) $[x, x+1] \times [y, y+1]$. We say $S_{x,y} \prec S_{x',y'}$ if one of the following two conditions holds:
\begin{itemize}
	\item[(1)] $x + y = x' + y'$ and $x < x'$;
	\item[(2)] $|(x+y)-(x'+y')| = 1$, $x \leq x'$ and $y \geq y'$.
\end{itemize}
The `left-of' terminology makes sense if we rotate the region forty-five degrees counter clockwise so that square of a fixed level form horizontal lines; see Figure~\ref{FigTilingOrder4}. 

Let $T$ be a tile and $S$ be a square (or a boundary square). We write $S \prec T$ if $S \prec S'$ for some square $S' \in T$, and $T \prec S$ if $S' \prec S$ for some square $S' \in T$. If $T_1$ and $T_2$ are two tiles in a tiling, we write $T_1 \prec T_2$ if there exist a square $S_1 \in T_1$ and a square $S_2 \in T_2$ with $S_1 \prec S_2$. It is not possible that both $T_1 \prec T_2$ and $T_2 \prec T_1$ unless $T_1 = T_2$. 

For a region $R$ and a fixed ($n$-ribbon) tiling, let $\sigma_l$ and  $\tau_l$ be the number of squares of $R$ and tiles at level $l$, respectively. By the definition of ribbon tiling, we have 
\begin{equation}
\label{eqn:sigmatau}
	\sigma_l = \sum_{j=l-n+1}^{l}  \tau_j
\end{equation}
for each level $l$ in every tiling of $R$. Clearly, $\sigma_l$ does not depend on the tiling. Now $\tau_l=0$ if $R$ has no tiles at level $l$, and so the equation~\eqref{eqn:sigmatau} shows that $\tau_l$ does not depend on the tiling, only on the region $R$. So each tiling of $R$ has the same number of tiles in a specific level. We can order tiles of level $l$ from left to right, and denote the $i$-th tile of level $l$ in this ordering as $T^{(l)}_i$, $i = 1, \ldots , \tau_l$.
%This enumeration does not depend on tilings. 
%Similarly, we use the notation $S^{(l)}_i$ to represent the $i$-th square at level $l$ in a region.

\section{Proof of Theorem \ref{ThmUpBound}}	\label{SecProof}
In this section, we will always assume that the region $R$ is rotated forty-five degrees counter clockwise and that the lowest level of a square in $R$ is $0$. We will often write tiling to mean $n$-ribbon tiling.

We first show that a tiling is uniquely determined by the positions of root squares.
\begin{lemma}	\label{LemRootSquareDetermination}
	Let $R$ be a finite region of the plane. Any $n$-ribbon tiling of $R$ is determined uniquely by the positions of the root squares of the $n$-ribbons in the tiling.
\end{lemma}

\begin{proof}
	Fix a subset $I \subset R$ of squares that are the root positions of our
tiling. We show that the whole $n$-ribbon tiling of $R$ may be deduced from
$I$, by building the tiling from the low-level squares of $R$ upwards.

All the squares in $R$ of level $0$ are root squares of (distinct) $n$-ribbons. (In other words, $I$ must contain all these squares.) So we have no choice for the intersection of the
tiling with the squares of level $0$. 
Suppose now (as an inductive hypothesis) that we have found the intersection of the tiling with all squares in $R$ of level $l$ or less, for some integer
with $l \geq 0$. We claim that there is only one choice for the intersection of our tiling with the squares in $R$ of level $l+1$ or less. 

Certainly the tiling for those level $l+1$ squares that are root squares of $n$-ribbons are determined, as they are exactly the set of squares of level $l+1$ in $I$. The remaining squares of level $l+1$ are covered by the set $T$ of tiles in our tiling of level $k$ where $l-(n-1)+1 \leq k \leq l$, with one square of level $l+1$ covered for each tile in $T$. The tiles in $T$ may be ordered from left to right, by the order in which we meet them as we move rightwards along the squares of level $l$. Because no two $n$-ribbons in our tiling cross, this order does not change if we instead order $T$ by moving along the squares of level $l+1$. But this means that the tiling at level $l+1$ is determined: the $i$-th tile in $T$ covers the $i$-th square of level $l+1$ in $R \setminus I$, moving from left to right. So our claim follows.

The lemma now follows by induction on $l$.
\end{proof}

We observe that there is a straightforward upper bound on the per-tile entropy of a region $R$ of area $A$ as a direct consequence of Lemma~\ref{LemRootSquareDetermination}, which can be derived as follows. Any $n$-ribbon tiling of $R$ contains $A/n$ tiles, and so the set $I\subseteq R$ of root tiles in the tiling satisfies $|I|=A/n$. Hence the number of possibilities for $I$ is at most $\binom{A}{A/n}$ and hence (using a standard upper bound for binomial coefficients)
\[
\mathrm{Ent}_{n}(R)\leq \frac{\log_2 \binom{A}{A/n}}{A/n}\leq \frac{\log_2\, (Ae/(A/n))^{A/n}}{A/n}=\log_2 n+\log_2 e.
\]
This upper bound is weaker than Theorem~\ref{ThmUpBound}, but is still reasonable.

To prove Theorem~\ref{ThmUpBound}, we require more information on the structure of the subsets $I$ of root squares. We will proceed inductively, considering the choices for the root squares of level $l$, once the root squares at level $k$ with $k < l$ have been determined. We examine the squares in $R$ of level $\ell-1$ (where our tiling is determined) and level $l$ (where our root squares must lie), together with the boundary squares of levels $\ell-1$ and $\ell$. We show that the root squares at level $l$ must lie in certain disjoint subsets $A_i^{(l)}$ of squares in $R$ at level $l$. The subsets $A_i^{(l)}$ are determined by the root squares at level $k$ with $k < l$. We now provide an argument which defines the subsets $A_i^{(l)}$, and shows their relationship with root squares. (See Lemma~\ref{LemThreeForms} below.)

Suppose all root squares at level $k$ with $k < l$ have been determined. The proof of Lemma \ref{LemRootSquareDetermination} shows that the tiling is determined on all squares of level less than $l$. In particular, the tiles at level $l-n$ are completely determined. Let $\mathbf{D}^{(l-1)}$ be the set of end squares of the tiles at level $l-n$. Clearly, $\mathbf{D}^{(l-1)} \subseteq \mathbf{S}^{(l-1)}$, and $\mathbf{D}^{(l-1)}$ has been determined.

Let $\mathbf{W}^{(l)} = (\mathbf{S}^{(l-1)} \setminus \mathbf{D}^{(l-1)}) \cup \mathbf{S}^{(l)}$ and $\mathbf{K}^{(l)} = \mathbf{D}^{(l-1)} \cup (\mathbf{B}^{(l-1)} \cup \mathbf{B}^{(l)} )$. It is not difficult to check that $\mathbf{W}^{(l)} \cup \mathbf{K}^{(l)} = \bigcup_{k=l-1,l} (\mathbf{S}^{(k)} \cup \mathbf{B}^{(k)} )$ and $\mathbf{W}^{(l)} \cap \mathbf{K}^{(l)} = \emptyset$.
The squares in $\bigcup_{k=l-1,l} (\mathbf{S}^{(k)} \cup \mathbf{B}^{(k)} )$ are linearly ordered by the `left of' relation defined in the previous section, and so the set consists of \emph{runs} of consecutive squares in $\mathbf{W}^{(l)}$, separated by one or more squares in $\mathbf{K}^{(l)}$. (Alternatively, we can think of $\mathbf{W}^{(l)}$ being divided into equivalence classes by their `left-of' relations with $\mathbf{K}^{(l)}$ such that the squares in each equivalence class have the same `left-of' relation with every element in $\mathbf{K}^{(l)}$. Clearly, the equivalence classes are exactly the runs above.)

In order to better understand the runs in $\mathbf{W}^{(l)}$, we think of the squares in $\bigcup_{k=l-1,l} (\mathbf{S}^{(k)} \cup \mathbf{B}^{(k)} )$ as being coloured black or white: the squares in $\mathbf{K}^{(l)}$ are black and those in $\mathbf{W}^{(l)}$ are white. Up to left-right reflection, each run will have one of three forms (a), (b) or~(c): see Figure~\ref{FigThreeForms}. Let $d$ be the difference of the number of squares at level $l$ and $l-1$ in a run. Then the three forms (a), (b), (c) correspond to the case $d=-1,0,1$, respectively.

\begin{figure}[H] 
	\centering
	\includegraphics[scale=0.35]{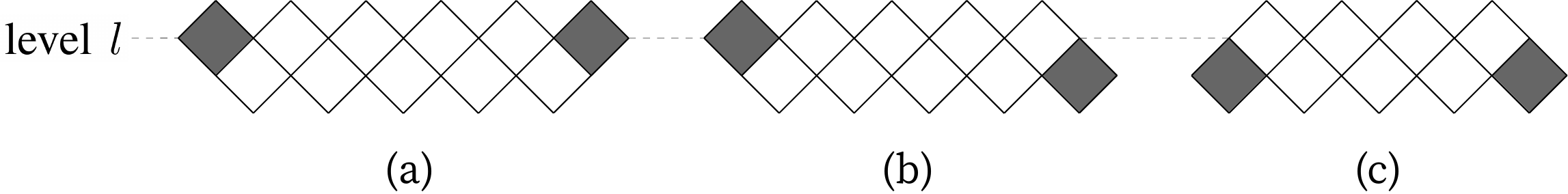} 
	\caption{Three forms of a run. Black squares are either (i) the  end squares of a tile at level $l-n$ or (ii) the boundary squares at level $l$ or $l-1$.}  
	\label{FigThreeForms}
\end{figure}

\begin{lemma}	\label{LemThreeForms}
	A run of the form (a) cannot be part of a tiling. A run of the form (b) does not contain any root squares of level $l$. A run of the form (c) contains exactly one root tile of level $l$.
\end{lemma}

\begin{proof}
For a fixed run, let $S^{(k)}_i$ be the $i$-th square at level $k$ with $k=l, l-1$.

First, suppose we have a run of the form (b). By the definition of our black-white colouring, it follows that the square $S^{(l-1)}_1$ is covered by a tile $T_1$ whose level is larger than $l-n$ and less than $l$. Thus, the square $S^{(l)}_1$ is also covered by the tile $T_1$. By repeating this argument, we have $S^{(l-1)}_i$ and $S^{(l)}_i$ must be covered by the same tile $T_i$ for all $i$. Therefore, a run of the form (b) does not contain any root squares at level $l$. 
	
Using similar argument for runs of the form (a), we see that the last square at level $l-1$ must be the end square of a tile at level $l-n$, so it must be black, which contradicts the definition of the form (a).
	
Finally, suppose our run has the form (c). Since the number of white squares of level $l$ exceeds that of level $l-1$ by $1$, there is exactly one root square of level $l$ in the run.
\end{proof}

From Lemma \ref{LemThreeForms}, it follows that the number of runs of the form~(c) is equal to the number of tiles at level $l$ in any tiling. Indeed, 
the root square of tile $T^{(l)}_i$ lies in the $i$th run of the form (c) (reading left to right), which we denote $E^{(l)}_i$. Let $A^{(l)}_i \subset E^{(l)}_i$ be the set of level $l$ squares in $E^{(l)}_i$. The root square of tile $T^{(l)}_i$ must be chosen from $A^{(l)}_i$, as it is of level $l$.

An example of this situation is depicted in Figure \ref{FigPartition}. In this example, the tiling is determined on squares of level $5$ or less. The set $\mathbf{W}^{(6)}$ is separated into five runs: one of form (b) and four of form (c). So there must be four level $6$ root squares; the $i$th root square is contained in the set of level $6$ squares in the $i$th run of the form~(c).

\begin{figure}[H] 
	\centering
	\includegraphics[scale=0.45]{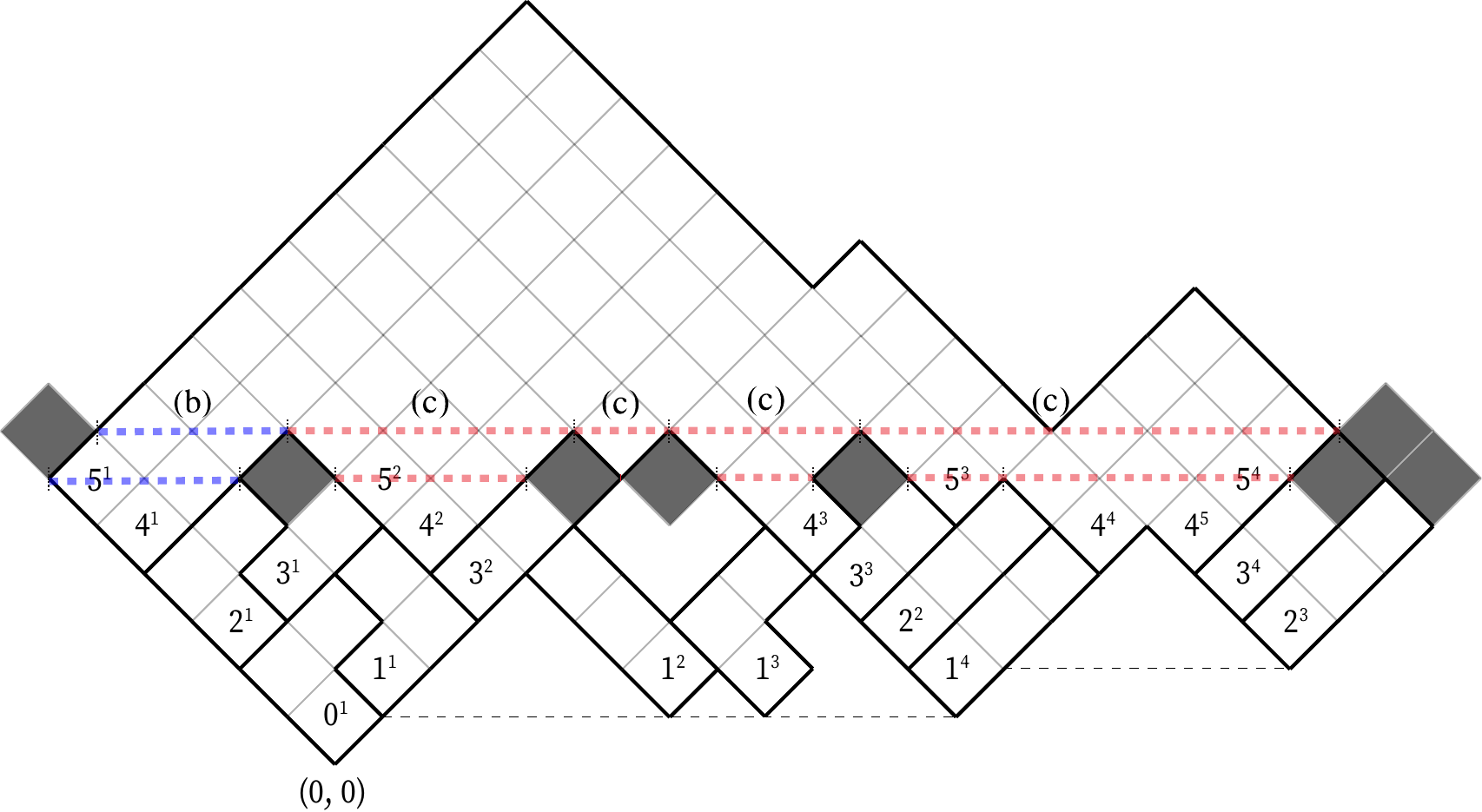} 
	\caption{An example with $n=3$ and $l=6$. The determined root squares are marked as $k^i$ for the $i$-th tile at level $k$ for $k < l$. The tiles whose level are lower than $6$ have been determined. The black squares of $\mathbf{K}^{(l)}$ split $\mathbf{W}^{(l)}$ into five runs indicated by dashed segments at levels $5$ and $6$, including one run of the form (b) and four of the form (c). Each run $E^{(l)}_i$ of the form (c) corresponds to the tile $T^{(l)}_i$, $i=1,2,3,4$.}  
	\label{FigPartition}
\end{figure}

Note that our black-white colouring, and so the sets $A^{(l)}_i$, is completely determined by the tiling on squares of level $l-1$ or less.

%Now we are ready to prove Theorem \ref{ThmUpBound}. 
\begin{proof}[Proof of Theorem 1]
From Lemma \ref{LemRootSquareDetermination}, it follows that for each level $l$ we need to choose $\tau_l$ squares as root squares from the set $\mathbf{S}^{(l)}$ to construct a tiling. Let $L$ be the highest level of the tiles in $R$. We choose the root squares in each level from $0$ to $L$ in turn. Once the root squares at level $l-1$ and below are chosen, we first construct the disjoint sets $A^{(l)}_i$, $i=1,2,\cdots,\tau_{l}$. Then we choose one root square from each set $A^{(l)}_i$. All tilings arise in this way, by Lemma \ref{LemThreeForms}.

Let $\mathcal{I}_l$ be the number of possible choices of root squares at level $l$ as we construct a tiling. We have
\begin{equation*}	
	\mathcal{I}_l = \prod_{i = 1}^{ \tau_l } | A^{(l)}_i | .
\end{equation*}
It is clear that the cardinalities $| A^{(l)}_i |$ satisfy the constraint $\sum_{i = 1}^{ \tau_l } | A^{(l)}_i | \leq | \mathbf{S}^{(l)} |$. 

Let $x^{(l)}_i$, $i=1,2,\cdots,\tau_l$, be positive integers satisfying $\sum_{i = 1}^{ \tau_l } x^{(l)}_i \leq | \mathbf{S}^{(l)} |	$ for every level $l = 1,\cdots,L$. Note that the constraints on the integers $x^{(l)}_i$ do not depend on the tiling in any way, just the region $R$. Maximizing $\prod_{i = 1}^{ \tau_l } | A^{(l)}_i | $ over all possible choices of root squares whose levels are lower than $l$, we have 
$$
	\mathcal{I}_l \leq \max \Big( \prod_{i = 1}^{ \tau_l } | A^{(l)}_i | \Big)
						 \leq \max \Big( \prod_{i = 1}^{ \tau_l } x^{(l)}_i \Big)	.
$$
Since the constraints on the integers $x^{(l)}_i$ do not depend on the tiling, it follows that
$$
	\mathcal{T}_n(R) 
	\leq \prod_{l = 1}^{ L } \mathcal{I}_l 
	\leq \max \Big( \prod_{l = 1}^{ L } \prod_{i = 1}^{ \tau_l } x^{(l)}_i \Big)	.
$$

Therefore, the solution of the following maximization problem is an upper bound on $\mathcal{T}_n(R)$:
\begin{equation}		\label{EqMaxProblemLevel}
	\begin{aligned}
		\text{maximize} & \quad \prod_{l = 1}^{L}  \prod_{i = 1}^{\tau_l} x^{(l)}_i 	\\
	 \text{subject to:} & \quad \sum_{i = 1}^{ \tau_l } x^{(l)}_i  \leq | \mathbf{S}^{(l)} |  \quad \text{ for } l = 1, 2, \cdots, L,	\\
	& \quad x^{(l)}_i > 0 \quad \text{ for }l = 1, 2, \cdots, L \text{, }\,i=1,2,\ldots,\tau_l.
	\end{aligned}
\end{equation}	
For integers satisfying the constraints in this problem, we have
\[
\sum_{l=1}^{L} \sum_{i = 1}^{\tau_l } x^{(l)}_i \leq \sum_{l=1}^L  | \mathbf{S}^{(l)} |=A(R).
\]
Hence the following problem has weaker constraints than (\ref{EqMaxProblemLevel}):
\begin{equation}		\label{EqMaxProblem}
	\begin{aligned}
		\text{maximize} & \quad \prod_{l = 1}^{L}  \prod_{i = 1}^{\tau_l} x^{(l)}_i  	\\
	    \text{subject to:}& \quad \sum_{l=1}^{L} \sum_{i = 1}^{\tau_l } x^{(l)}_i  \leq A(R),\\
			  & \quad x^{(l)}_i > 0 \quad \text{ for }l = 1, 2, \cdots, L \text{, }\,i=1,2,\ldots,\tau_l .
	\end{aligned}
\end{equation}	
Since the constraints in (\ref{EqMaxProblem}) are looser than those in (\ref{EqMaxProblemLevel}), it follows that the solution of (\ref{EqMaxProblem}) is an upper bound on $\mathcal{T}_n(R)$.

We know that $\sum_{k=1}^{L} \sum_{i = 1}^{\tau_l } 1 = \frac{A(R)}{n}$ is the number of tiles in $R$. The maximum in (\ref{EqMaxProblem}) is obtained by setting all $x^{(l)}_i$ equal to $A(R)/ ( \frac{A(R)}{n} ) = n$. Then, the solution of (\ref{EqMaxProblem}) is
$$
	\max \Big( \prod_{l = 1}^{L}  \prod_{i = 1}^{\tau_l} x^{(l)}_i  \Big) = n^{\frac{A(R)}{n}}	.
$$

Hence, $\mathcal{T}_n(R) \leq
 n^{\frac{A(R)}{n}}$, and $\mathrm{Ent}_n(R) \leq \log_2 n$ as required.
\end{proof}

\end{document}